\documentclass[11pt,a4paper,reqno]{amsart}
\usepackage{amsmath}
\usepackage{amsfonts}
\usepackage{amssymb}
\usepackage{amscd}
\usepackage[british,UKenglish,USenglish,english,american]{babel}

\newcommand\R{{\mathbf{R}}}

\renewcommand\P{{\mathbf{P}}}
\newcommand\E{{\mathbf{E}}}

\newcommand\Z{{\mathbf{Z}}}

\newcommand\I{{\mathbf{I}}}

\newcommand\CL{{\mathcal L}}

\theoremstyle{plain}
  \newtheorem{theorem}[subsection]{Theorem}
  \newtheorem{conjecture}[subsection]{Conjecture}

  \newtheorem{corollary}[subsection]{Corollary}
\theoremstyle{remark}

\theoremstyle{definition}
  
\include{psfig}
\begin{document}
\title[]{Random walks with different directions: Drunkards beware !}

\author{Sim\~{a}o Herdade }
\address {Department of Mathematics, Rutgers Univ., Piscataway 08540}
\email{simaoh@math.rutgers.edu}
\thanks{ Simao Herdade was supported by Fundacao para a ciencia e tecnologia, grant award 65678/2009, financed by POPH - QREN}

\author{Van Vu}
\address{Department of Mathematics, Yale University, New Haven 06520}
\email{van.vu@yale.edu}
\thanks{V. Vu is supported by   NSF  grant DMS 1307797  and AFORS grant FA9550-12-1-0083.}

\begin{abstract} As an extension of Polya's classical result on 
random walks on the square grids ($\Z^d$), we consider a random walk where the steps, while still have unit length,  point to different directions. 

We  show that in  dimensions at least 4, the returning probability after $n$ steps is at most $n^{-d/2 - d/(d-2) +o(1) }$, which is sharp. 
The real surprise is in dimensions 2 and 3. 

In dimension 2, where the  traditional grid walk is  recurrent, our upper bound is $n^{-\omega (1) }$, which is much worse
than higher dimensions.  

In dimension 3, we prove an upper bound of order $n^{-4 +o(1) }$. We discover 
 a new conjecture concerning incidences between spheres and points in $\R^3$, which, if holds, would improve the bound  to $n^{-9/2 +o(1) }$, which is consistent 
to the $d \ge 4$ case.  This conjecture  resembles Szemer\'edi-Trotter type results and is of independent interest. 

\end{abstract}
\maketitle

\section{Introduction}

In his classical paper in 1921, Polya \cite{Polya} proved his famous
theorem on random walks on $\Z^d$. Consider a drunkard  at the origin (his
home) at time 0; at each tick of the clock he goes to a randomly selected neighbouring
lattice point,  uniformly at random. One is interested in the chance that
 he
 returns to home  at time $n$. Let $e_1, \dots, e_d$ be the basis unit
vectors in $\Z^d$ and $\xi_i$ are iid Bernoulli random variables (taking
values $\pm 1$ with probability $1/2$).
 Let

 $$S_n :=  \sum_{j=1} ^n \xi_j f_j  $$ where $f_j $ is chosen uniformly
from $E := \{e_1, \dots, e_d \}$.  One is  interested in  estimating

 \begin{equation} \P( S_n = 0) . \end{equation}
 
  Polya proved

 \begin{theorem}  \label{Polya}
 For any $d \ge 1$, $\P (S_n =0)  = \Theta ( n^{-d/2} ). $

 \end{theorem}

 (In general, we can consider the drunkard starting from an arbitrary
address.  For the sake of presentation,  we delay the discussion of this
case  until the end of the section.)
  A random walk is said to be {\it recurrent }
 if it returns to its initial position with probability one. A random walk
which is not recurrent is called {\it transient. }
 Theorem \ref{Polya} implies

\begin{corollary}
The simple random walk on $\Z^d$  is recurrent in dimensions $d=1,2$  and
transient in dimension $d \ge 3$.
\end{corollary}
As well known, it  is good news for our drunkard. If he lives on a
square lattice, at least !

\vskip2mm
The goal of this note  is to show that once the drunkard leaves the
lattice, life is no longer  rosy.
As  a matter of fact,  dimension 2 (where he actually moves  in) turns out
to be the worst.

In what follows, we  let the drunkard walk with rather general directions.
Technically,  at each tick of the clock, he chooses a vector, and decides
to walk along its direction or in the opposite one.
The new  and critical assumption here is that the vectors are not to be
repeated. (One can imagine, for example, that our drunkard is walking on a
big plaza and he himself decides his next step.)
Mathematically, we consider the random walk $$S_{n, V}=\eta_1 v_1 + \eta_2
v_2 + \cdots + \eta_n v_n$$
  where $V := \{ v_1, v_2, \cdots, v_n \} $ is a set of $n$ different unit
vectors in $\mathbb{R}^d$, and
$\eta_i$ are again i.i.d. Bernoulli random variables.
 We say that $V$ is {\it effectively} $d$-dimensional if there is no
hyperplane that contains more than $.99 n$ vectors in $V$ (where $.99$
can be replaced by
any constant less than 1).

Here are the (bad) news

\begin{theorem} \label{main1}   Consider a set $V$ of $n$ different unit
vectors which is effectively $d$-dimensional. Then
\begin{itemize}
\item For $d \ge 4,$
$ \P (S_{n, V} =0)  \le n^{-\frac{d}{2}-\frac{d}{d-2}+ \ o(1)}. $
\item For $d=3$, $\P(S_{n, V}=0 ) \le   n^{-4 +o(1)}  . $
\item For $d=2$, $\P (S_{n, V} =0)  \le n^{-\omega (1) }. $ \end{itemize}
\end{theorem}

The assumption "effectively $d$-dimensional"  is necessary, otherwise one
can take $V$ in a lower dimensional subspace and have a better bound.
The bound for $d \ge 4$ is sharp, as we can construct a set $V$ such that
$  \P (S_{n, V})  \ge n^{-\frac{d}{2}-\frac{d}{d-2}+ \ o(1)}$.
We conjecture that this is the sharp bound in the case $d=3$ as well.
This conjecture would follow a new conjecture concerning incidences in
$\R^3$, which is of independent interest (see Section
\ref{incidence problem} for details).

The real bad news for the drunkard is the case $d=2$, where no matter how
he chooses the set $V$, the returning probability is super polynomially
small. Deciding the order of the exponent is an interesting problem. We
can construct a set $V$ which provides  $\P( S_{n, V} =0) \ge n^{- C \log
\log n}$ for some constant $C >0$.

Theorem \ref{main1} holds under a weaker assumption. We can allow the
vectors in $V$ to take different lengths and also have some
multiplicities.
We say that $V$ is $(L,M)$-typical  if the vectors of $V$ have lengths in
a set $\CL$ of size $L$, and each vector has multiplicity at most $M$. 
Furthermore, we can allow
the target to be any point $x \in \R^d$.  (This corresponds to a walk
starting at $-x$ and ending at the origin.)

\begin{theorem}  \label{main2}
Let $V$ be a $(L, M)$-typical set  which is effectively $d$-dimensional,
where both $L, M =n^{o(1) }$. Then
the upper bounds in Theorem \ref{main2} holds for  the probability $\P (
S_{n, V } = x )$, for any $x \in \R^d$.
\end{theorem}
 
In the next section, we present our main lemmas. The proofs of the
theorems follow in Section \ref{argument}. We conclude with an open
problem in incidence geometry which would imply the sharp bound in the
case
$d=3$.

\section{The main lemmas} \label{section:lemmas}

In this section, we describe our main tools.  Let $G$ be an abelian group.
A generalized arithmetic progression (GAP) in $G$  is a set of the form
 $$ Q(a_0, a_1,\cdots, a_r, N_1,\cdots,N_r)=\{a_0+ x_1 a_1+ \cdots + x_r
a_r|  0 \le x_i \le N_i \} .$$
 
We refer to  $ r$ as the rank, and call   $a_0,a_1, \cdots, a_r \in G$
the generators, and  $N_1, \dots, N_r$  the dimensions of $Q$.
It's generally useful to think of a generalized arithmetic progression, as
the image of the discrete $r$-dimensional box $ [0,N_1]\times \dots
\times[0,N_r]$ under the map

  \begin{align*}
   \Phi : [0,N_1]\times \dots \times[0,N_r]\, &\longrightarrow  \ G
   \\   ( x_1,\cdots , x_r) \ &\longrightarrow a_0 + a_1 x_1 + \cdots +
a_r x_r
\end{align*}

We say that $Q$ is proper if the above map is injective.
We say that $Q$ is symmetric if it can also be written in the form
\newline $$Q=\{ n_1 b_1 + n_2 b_2 + \cdots + n_r b_r :\ -M_i\leq n_i \leq
M_i,  i=1\ldots r \}$$ 
for some $b_1,\ldots, b_r \in \mathbb{R}^d$ and
$M_1,\ldots, M_r \in \mathbb{N}$.

Let $V$ be a set of $n$ vectors, define the concentration probability

 \begin{equation} \label{rho} \rho (V)= \displaystyle\sup_{\substack{x \in
\mathbb{R}^d }} P (S_{n, V} =x).  \end{equation}
 
We are going to use the following result  in \cite{InverseLO}, which
asserts that a set of vectors with
 high concentration probability must necessarily be, up to a few elements,
a subset of a generalized arithmetic progression of small cardinality.

\begin{theorem}(Optimal inverse Littlewood-Offord theorem) \label{LO}
Let $\epsilon < 1$ and $C$ be positive constants. Assume that $\rho(V)
\geq n^{-C}$.
Then, there exists a proper symmetric GAP $Q$, of some rank $r=O_{C,
\epsilon}(1)$, that contains at least $(1-\epsilon)n$ elements of $V$,
such that $$|Q|=O_{C, \epsilon}(\rho(V)^{-1}n^{-\frac{r}{2}})$$
\end{theorem}
 
Our next tool is a result of Chang  \cite{chang}.  For a set $X$ and a
number $m$, both in the complex plane, denote
by $r_2 (m ; X)$  the number of ways to write $m$  as a product of two
elements of $X$.

\begin{theorem}        \label{chang} For any fixed $r$ there  is some constant
$C_r>0$ such that the following holds.
Let $Q$ be a GAP of complex numbers of rank $r$ and dimensions $N_1,
\cdots, N_r$.
Let $N=\max_i N_i$.  Then for all $m \in \mathbb{C} $,
$$ r_2 (m ; Q) \le  N^{\frac{C_r}{\log \log N}} $$
\end{theorem}

We use this theorem to prove the following corollary.
\begin{corollary} \label{circle}
Let $Q$ be a GAP in $\mathbb{R}^2$ with constant rank $r$ and cardinality
$n$.
Let  $S \subseteq \mathbb{R}^2 $ be an arbitrary  circle. Then $$|Q\cap S|
\leq n^{o(1)}.$$
\end{corollary}

\begin{proof}  (Proof of Corollary \ref{circle} )
For  any $x\in \mathbb{C}$ denote its complex conjugate by $\bar{x}$.
Let $S$ be a circle of radius $R$.  As we can shift $Q$, we can assume,
without loss of generality, that $S$ is centered at $0$.
Let $Q=\{a_0+ x_1 a_1+ \cdots + x_r a_r| x_i=0,\cdots,N_i, \forall i \}
\subseteq \mathbb{C}$, with $|Q|= n$.
 Consider \ $$P=\{x_0 a_0 + x_1 a_1+ \cdots + x_r a_r + y_0\bar{a}_0 + y_1
\bar{a}_1+ \cdots + y_r \bar{a}_r|  \,x_0, y_0\in\{ 0,1\}; |x_i|, |y_i|
\le N_i \}.$$
By the above theorem

  $$ r_2 (R^2 ; P) \le  N^{\frac{C_{2r+2}}{\log \log n}} \le
n^{\frac{C_{2r+2}}{\log \log n}}=n^{o(1)} .$$

 On the other hand,  $ x\in S $ if and only if $ x\cdot\bar{x}=  R^2$. As
$P$ contains all elements of $Q$ and their conjugates,
 it follows that $$|Q\cap S| \leq   r_2 (R^2 ; P)= n^{o(1)} .$$ \end{proof}

\section{Proof of Theorems \ref{main1}  and \ref{main2} } \label{argument}
In the first two subsections, we prove Theorem \ref{main1}.
\subsection{ $d \ge 4$: Upper bound }
 
Consider a set $V$ and assume, for a contradiction, that $\P ( S_{n, V }
=0 ) \ge n^{-d/2 - d/(d-2) + \delta} $ for some constant $\delta >0 $.
By Theorem \ref{LO}, there is a proper symmetric GAP $Q$ of constant  rank
$r$,  which contains at least $.99 n$ elements of $V$, and
 
\begin{equation}\label{smallQ}  |Q| = O( n^{d/2 +d/(d-2) -r/2 -\delta } ).
\end{equation}

In what follows, we derive a lower bound that contradicts \eqref{smallQ} .
Let  $Q :=\{ n_1 a_1 + n_2 a_2 + \cdots + n_r a_r;  |n_i| \le N_i \} $,
$Q' := \{ n_3 a_3 + n_4 a_4 + \cdots + n_r a_r;  |n_i| \le N_i \} $ and
$Q^{''}:= \{ n_1 a_1 +n_2 a_2. |n_i | \le N_i \}$. We can assume, without
loss of generality, that $N_1, N_2$ are the two largest dimensions, which
implies that $|Q'| \le |Q|^{(r-2) /r} $.

By Corollary \ref{circle} and the fact  that the vectors in $V$ have unit
length, we conclude that for any $x \in Q'$, $|(x + Q^{''} ) \cap V |
\leq n^{o(1) }$. Since $V = \cup_{x \in Q'} (x+ Q^{''} ) \cap V$, it
follows that
$n \le n^{o(1)}  |Q|^{(r-2) /r } $, or equivalently, $|Q| \ge n^{ r/(r-2)
-o(1) }$.   Together with \eqref{smallQ}, we have

$$ n^{d/2 +d/(d-2) -r/2 -\delta }   \ge n ^{r/(r-2) -o(1) }. $$

On the other hand $V$ is effectively $d$-dimensional, so $r \ge d$. For $r
\ge 4$, the function $f(r)= r/2 + r/(r-2) $ is strictly monotone
increasing. This implies that the above
inequality cannot hold for sufficiently large $n$, a contradiction.
 
\subsection{  $d=3$: Upper bound}
One can repeat the above argument, but we can no longer use the fact that
$f(r)= r/2 + r/(r-2)$ is monotone.  As a matter of fact
$f(3) =9/2$ is large than both $f(4) =4$ and $f(5)= 25/6$. As $f(r) \ge 5$
for all $r \ge 6$, the worst value one can take is $f(4)=4$, which results
in the upper bound
$n^{-4 +o(1) }$.
 
\subsection{$d=2$: Upper bound }
Consider a set $V$ and assume,  for a
contradiction, that $ P(S_{n,V}=0) \geq \frac{1}{n^C}$,  for some constant
$C$.
By Theorem \ref{LO}, there is  proper symmetric GAP $Q$, of some  constant
rank  $r=O_{C, \epsilon}(1)$, that contains at least $(1-\epsilon)n$
elements of $V$,  and with $$|Q|=O_{C,
\epsilon}(\rho(V)^{-1}n^{-\frac{r}{2}})$$
However, by Corollary \ref{circle}, such $Q$ can only contains $|Q|^{o(1)
}$ points from the unit circle, which, in turns, contains $V$. This
provides the desired contradiction.

\subsection{ Lower bounds }
Let us start with the case $d \ge 3$. We construct a set $V$ such that

$$ \P( S_{n, V} = 0) \ge n^{-d/2 - d/(d-2) -o(1) } . $$

By classical results on Waring's problem \cite{Vaughan}, the number of ways to write an
integer $N$ as sum of $d$ squares is at least $n := N^{ (d-2)/2 +o(1) } $,
for any fixed $d \ge 4$ and all sufficiently large $N$. This
means the sphere  of radius $R := N^{1/2} $ (centered at the origin)
 contains at least $n$ lattice vectors. Let $V$ be the set of these
vectors (we can normalize them to have unit length).  An application of
the central limit theorem shows that with probability
at least $1/2$, $S_{n,V}$ belongs to  the ball  $B$ of radius $10 n^{1/2}
R$ centered at the origin.  Thus, there is a lattice point $x$ in this
ball such that
$$ \P( S_{n, V} =x) \ge \frac{1}{2} ({\rm volume } \,\, B ) ^{-1} \ge C
n^{-d/2 - d/(d-2) -o(1) } $$ for some positive  constant $C = C(d) $.

One can show that the supremum $\sum_{x } \P(S_{n,V} = x )$ is attained at
$x=0$, for any set
$V$ symmetric with respect to the origin. We use  Gauss' identity  $$ \I
_{Y =0} = C_d \int_{ S^{d-1}  } e ( Y \cdot t ) dt, $$  where $Y$ is a
vector in $\R^d$,
$\I$ is the indicator function,  $C_d$ is a positive constant depending on
$d$, $e (x) =\exp ( 2\pi i x ) $ and $S^{d-1} $ is the unit sphere in
$\R^d$.  By this identity, we have

$$\P (S_{n, V} =x) = \E \I_{S_{n, V} -x =0 } =  \E C_d \int_{S^{d-1} }  e
( (S_{n,V } -x ) \cdot t)  dt  = C_d \int_{S^{d-1}  } e (-x \cdot t) \E e
(S_{n, V } \cdot t ) dt . $$

As $S_{n, V} = \sum_{i=1}^n \eta_i v_i$ where the $\eta_i$ are
independent, it follows that

$$ \E e (S_{n, V } \cdot t ) =\prod_{i=1}^n \E e ( \eta_i v_i \cdot t ) =
\prod_{i=1}^n  \cos ( v_i \cdot t) . $$

Since the set $V$ is symmetric with respect to the origin, $$\prod_{i=1}^n
 \cos ( v_i \cdot t)  = \prod_{i=1}^n | \cos (v_i \cdot t ) | . $$  
Thus, by the triangle inequality

$$ \P( S_{n,V} =x ) \le   C_d \int_{S^{d-1} }  \prod_{i=1}^n | \cos (v_i
\cdot t ) |  dt = \P ( S_{n, V} =0 ) $$ for any $x \in \R^d$.

Let us now turn to the case $d=2$. Classical results in number theory show
that there are infinitely many
$R$ such that the circle centered at the origin of radius $R$ contains at
least $R^{1/ \log \log R} $  integral points \cite{2squares}.
 Let $V$ be the set of these  points,  it is easy to see that

$$\rho (V) = \Omega (  R^{-2 +o(1) }  )  \ge |V| ^{- C  \log \log |V| } $$
for a properly chosen constant $C$.

\subsection{Proof of Theorem \ref{main2} }  Assuming for a moment that $V$
consists of  different  vectors of unit length, the proof for an arbitrary
target $x$ is the same, since
in Theorem \ref{LO} we define $\rho (V) :=\sup_x \P( S_{n,V} =x )$.  For
the general case, by the pigeon hole principle,  there are at
least $n/LM$  different vectors in $V$ with
the same length
$t$. Let $V'$ be the set of these vectors and repeat the proof for this
set, conditioning on the rest of the walk. By the condition on $L, M$,
$|V'| = n^{1- o(1) }$ and this only influences the $o(1)$ terms in the
bounds.
 
\section{ New problems in incidence geometry} \label{incidence problem}

We conjecture that the upper bound $n^{-d/2 - d/(d-2) +o(1)}$ also holds
in the case $d=3$.  This would follows from the following conjectures,
which are of independent interest, 

\begin{conjecture}\label{triple sumset}
Let $V$ be a set of $n$ unit vectors in the euclidean space, with at most  $n^{o(1)}$  of its endpoints on
any plane. Then
$|V+V+V| \ge n^{5/2  -o(1) } $.
\end{conjecture}

\begin{conjecture}
Let $P$ be a set of $p$ points and $B$ be a set of $n^2$ unit spheres  in
$\R^3$. Again assume that no plane contains more than $n^{o(1)}$ points.
Assume as well that each sphere in $B$ contains $n$ points from $P$. Then
$p \ge n^{5/2-o(1)} $.
\end{conjecture}

As a matter of fact, we feel that one can replace both exponents $5/2$ by
$3$ (which would be clearly optimal).

Notice that the second statement implies the first. By congruence of triangles, the endopoints of all pairs of vectors with a prescribed sum lie in a same hyperplane. Under the hypothesis above, that means the size of $V+V$ is at least $n^{2-o(1)}$. Since each element of the triple sumset of V lies in a unit sphere centered at one of those $n^{2-o(1)}$ points, the conclusion follows.

It is also easy to see that conjecture \ref{triple sumset} implies the desired upper bound for the unproved cases in dimension $d=3$.
 We must show that if the generalized arithmetic progression containing all but a few elements of $V$ has rank $4$, or $5$, then its size is at least $n^{5/2-o(1)}$.

 Denote by $V'$ the set of $(1-\epsilon)n$ elements of $V$ contained in that GAP $Q$. The elements of $V'$ on any given hyperlane lie in the intersection of a circle with the projection of $Q$ onto the plane. This is a GAP of rank and size no greater than $Q$. By Corollary \ref{circle}, we conclude that $V'$ has at most $n^{o(1)}$ elements on any hyperplane. Furthermore we have that $V'+V'+V' \subseteq Q+Q+Q$.
 Assuming conjecture \ref{triple sumset}, this implies that $|Q+Q+Q| \ge n^{5/2  -o(1) } $ and so, since Q is a generalized arithmetic progression of constant rank,  its size itself is at least  $n^{5/2-o(1)}$.

\text{}

\end{document}